\documentclass[12pt]{amsart}

  \baselineskip 1 mm

\usepackage[osf,sc]{mathpazo}
\usepackage{amsmath,amsthm,amsfonts,latexsym,amsopn,verbatim,amscd,amssymb}
\usepackage{hyperref}

\theoremstyle{plain}
\newtheorem{thm}{Theorem}[section]

\newtheorem{prop}[thm]{Proposition}
\newtheorem{cor}[thm]{Corollary}
\theoremstyle{definition}

\newtheorem{remark}[thm]{Remark}
\newtheorem{definition}[thm]{Definition}
\newtheorem{Example}[thm]{Example}
\numberwithin{equation}{section}

\newcommand{\bnum}{\begin{enumerate}}
	\newcommand{\enum}{\end{enumerate}}

\begin{document}
\title{Representations and primitive central idempotents of a finite solvable group}
\author[Ravi S. Kulkarni and Soham Swadhin Pradhan]{Ravi S. Kulkarni and Soham Swadhin Pradhan}
\address{Bhaskaracharya Pratishthana, 56/14, Damle Path, Pune 411004, India}
\email{punekulk@gmail.com}

\address{Stat Math Unit, Indian Statistical Institute Bangalore, Bangalore 560059, India}
\email{soham.spradhan@gmail.com}
\begin{abstract}
Let $G$ be a finite solvable group. Then $G$ always has a  useful presentation, which we call a ``long presentation". 
Using a ``long presentation" of $G$, we present an inductive method of constructing the irreducible representations of $G$ over $\mathbb{C}$ and computing the primitive central idempotents of the complex group algebra $\mathbb{C}[G]$. For a finite abelian group, we present a systematic method of constructing the irreducible representations  
over a field of characteristic either $0$ or prime to order of the group and also a systematic method of computing the primitive central idempotents of the semisimple abelian group algebra.	
\end{abstract}
\maketitle
\section{Introduction}
Let $G$ be a finite group. Let $F$ be a field of characteristic either $0$ or prime to the order of $G$. We denote the order of $G$ by $|G|$. Let $\overline{F}$ be the algebraic closure of $F$. It is well known fact that the primitive central idempotents of the group algebra $\overline{F}[G]$ are all elements of the form 
$$\frac{\chi(1)}{|G|}\sum_{g \in G}\chi(g^{-1})g,$$
where $\chi$ is an irreducible $\overline{F}$-character of $G$. Using Galois decent, one can obtain all the primitive central idempotents of the semisimple group algebra $F[G]$. The known methods to compute the character table of a finite group are tasks of exponential growth with respect to the order of the group.
In view of the computational difficulty in this approach, mathematicians are interested in the problem of finding character-free methods for computing the primitive central idempotents of $F[G]$. In last few years, the problem has been solved for certain classes of groups and for some specific fields.

For a finite abelian group $G$, an explicit description of the primitive central idempotents of the rational group algebra and Wedderburn decomposition of $\mathbb{Q}[G]$ was considered by several authors (see \cite{jes 1}, \cite{milies126}). In this paper, for a finite abelian group $G$ and a field $F$ of characteristic either $0$ or prime to $|G|$, we present a systematic method of constructing the irreducible $F$-representations of $G$ and a character-free method of computing the primitive central idempotents of the group algebra $F[G]$.

Among the finite groups, the finite solvable groups surely cover a very large ground. Note that
\begin{enumerate}
	\item[1.] For almost all natural numbers $n$, all finite groups of order $n$ are solvable.
	\item[2.] Given a natural number n, almost all the groups of order n are solvable.
\end{enumerate}
Yet, the ground that is not covered, concentrated on a very small patch of ``measure zero", also contains many beautiful things. 

A solvable group always has a useful presentation, which we call a ``long presentation". In this paper, for a finite solvable group $G$, using a ``long presentation" of $G$ and the associated ``long system of generators", we iductively construct the irreducible representations of $G$ over $\mathbb{C}$ and also compute the primitive central idempotents of the group algebra $\mathbb{C}[G]$.
 More precisely, our main concern is avoiding the use of characters by a direct computation of the primitive central idempotents of $\mathbb{C}[G]$ in terms of a ``long system of generators" associated with a fixed maximal subnormal series of $G$, which refines it's derived series.
\section{Long Presentation}
For a finite solvable group $G$, we can always find a refinement of its derived series, which is a subnormal series of $G$ such that each successive quotients are cyclic groups of prime order. So, for a finite solvable group $G$ of order $N = p_{1}p_{2} \ldots p_{n}$, where $p_{i}$'s are primes, there is a subnormal series: $ \langle {e}\rangle = G_{o} \trianglelefteq G_{1} \trianglelefteq \ldots \trianglelefteq G_{n} = G$ such that for each $i = 1, 2, \ldots , n$, $G_{i}/G_{i-1} \approx C_{p_{i}}$, a cyclic group of order $p_{i}$ after a suitable ordering of $\{p_{1}, p_{2}, \dots , p_{n}\}$. One can choose an element $x_i$ in $G_{i}$ s.t. $x_i$ is a $p_{i}$-element of smallest order $p^{n_{i}}_{i}$, and $x_iG_{i-1}$ generates $G_i/G_{i-1}$. Note that such an element  always exists in $G$. Then $G$ has a presentation:
\begin{align*}
\langle \,x_1, x_2, \dots , x_n | x_i^{p_i^{n_i}} = 1, x_i^{p_i} = w_i(x_1, ..., x_{i-1}), \,x_i^{-1}x_jx_i = w_{ij}(x_1,..., x_{i-1}) \mbox{ for } j<i \,\rangle,
\end{align*}
where $w_i$ and $w_{ij}$ are certain words in $x_1, x_2, ..., x_{i-1}$. 
We call such a presentation, a 
{\it{long presentation}} of $G$, and $\{x_1, x_2, \dots , x_n\}$ is called the associated {{\it long system of generators}}. A remarkable thing is that each element of $G$ can be expressed uniquely as $x_1^{a_1}x_2^{a_2}...x_n^{a_n}$, where $0 \le a_i < p_i$.
\section{Representations and Primitive Central Idempotents of a Cyclic Group over an algebraically closed field}
Let $G = C_{N}$ be a cyclic group of order $N$. Let $N = {\prod_{p|N}}{p^{n_p}}$, be its factorisation into primes. Let $G_p = \{y \in G| {{y}^{p}}^{n_p} = 1\}$.
Then $G = {\prod_{p|N}}{G_p}$. Let $y_p$ be a generator of $G_p$. Then $y = {\prod_{p|N}}$ $y_p$ is a generator of $G$. Choose a system of long generators $G$ in each $G_p$. A system of long generators of $G$ is obtained by various products of long generators in each $G_p$.  
\begin{prop}\label{cyclic}
Let $G$ be a cyclic group of order $N$. Let $F$ be an algebraically closed field of characteristic $0$ or prime to $N$.
The irreducible representations of $G$ are the tensor products of the irreducible representations of $G_p$'s for $p |N$.
\end{prop} 
\begin{proof}
Let $\rho$ be an irreducible representation $\rho$ of $G$. As is well known, $\rho$ is of degree 1, and $\rho$ is given by $y \longmapsto \zeta$, where $\zeta$ is an $N^{th}$ root of unity. 
%
%
Clearly $G_p$ is a cyclic subgroup of order $p^{n_p}$. 
Let $\rho_p = \rho|_{G_p}$ be the restriction of $\rho$ to $G_{p}$. Suppose $\rho_p$ is a representation defined by $y_p \longmapsto \zeta_p$, where $\zeta_p$ is a $p^{n_p}$-th root of unity. Clearly $\zeta = \Pi_{p}{\zeta_{p}}$ and $\rho = \Pi_p\rho_p$.
Hence, $\rho$ is the tensor products of the irreducible representations of $G_{p}$'s for $p | N$.
\end{proof}
\begin{prop} Let $G$ be a cyclic group of order $N$. Let $F$ be a field of characteristic 0 or prime to $|G|$.
Let $\rho$ be an irreducible representation of $G$ over $F$.
Let $\rho_p$ = an irreducible component of $\rho|_{G_p}$, the restriction of $\rho$ to $G_{p}$. 
Then the primitive central idempotent corresponding to $\rho$ is the product of primitive central idempotents's  of $\rho_p$'s. 
\end{prop}
\begin{proof}
Let $\rho$ be an irreducible representation $\rho$ of $G$ is given by $y \longmapsto \zeta$, where $\zeta$ is an $N^{th}$ root of unity. 
The primitive central idempotent corresponding to $\rho$ in $F[G]$ is:
\begin{equation*}\tag{*}
e_{\rho} = \frac{1}{N}\sum_{i = 0}^{N-1} \zeta^{-i}y^i .
\end{equation*}
For simplicity we denote this expression by $e_{\zeta}$. 
The primitive central idempotent of $\rho_p$ is obtained by replacing $N$ by $p^{n_p}$ in $(*)$, and $y$ by $y_p$, 
Clearly the $e_{\rho} = e_{\zeta}$ factors as $\prod_{p | N}e_{\zeta_p}$. This completes the proof.
\end{proof}
\begin{thm}
Let $G = C_{p^{n}} = \langle{x_{1}, \dots ,x_{n} \mid x_{1}^{p} = 1, x_{2}^{p} = x_{1}, \dots , x_{n}^{p} = x_{n - 1}}\rangle$ be a long presentation of $G$.
Let $F$ be an algebraically closed field with characteristic $0$ or prime to $p$.
Then every primitive central idempotent of ${{F}}[G]$ can be expressed as
$$e_{{\zeta^{-1}_1}{x_{1}}}e_{{\zeta^{-1}_2}{x_{2}}} \dots e_{{\zeta^{-1}_n}{x_{n}}}$$ 
with $\zeta_{n}$ is a $p^{n}$th root of unity in ${F}$ and $\zeta^{p}_{n} = \zeta_{n - 1}, \dots, \zeta^{p}_{2} = \zeta_{1},$
where $e_{X} = (1+ X + \dots + X^{p-1})/{p}$ and $X$ is an indeterminate.
\end{thm}
\begin{proof}
Let $\rho$ be an irreducible representation $\rho$ of $G$, and  is given by $x_{n} \longmapsto \zeta$, where $\zeta$ is an $N^{th}$ root of unity. By Theorem \ref{cyclic}, $\rho$ is isomorphic to
the tensor product of $\rho_p$'s. The primitive central idempotent corresponding to $\rho$ in $F[G]$ is:
$$  e_{\rho} = \frac{1}{N}\sum_{i = 0}^{N-1} \zeta^{-i}{x_{n}}^i .$$
For simplicity we denote this expression by $e_{\zeta}$. 

Let $N = mn$, with $m> 1, n > 1$ be a factorisation in natural numbers. Then 
$$1 + X + \dots + X^{mn-1} = (1 + X + X^2 + \dots + X^{m-1})(1 + X^m + X^{2m} + \dots + X^{(n-1)m}).$$
We apply this repeatedly to the case: $N = p^n$. Write
$e_X = {1 + X + \dots + X^{p-1}}/p$, and let $ x_1, x_2, \dots, x_{n}$ be as in the statement, satisfy
$x_1^p = 1, x_2^p = x_1, x_3^p = x_2, \dots , x_n^p = x_{n-1}$.
Then writing $x_n = x$, we have ${1 + x + x^2 + ... + x^{p^n-1}}/p^n = e_{x_1}e_{x_2} \dots e_{x_n}$. Replacing $x$ by $\zeta^{-1}x$, we obtain the second universal factorisation of 
$$ e_{\zeta} = \displaystyle {\prod_{i = 0}^{n-1}e_{{\zeta_1}^{-1}x_1}e_{{\zeta_2}^{-1}x_2} ... e_{{\zeta_n}^{-1}x_n}},$$ where 
$\zeta_n = \zeta,  \zeta_{n-1} = \zeta^p, \zeta_{n-2} = \zeta_{n-1}^p = \zeta^{p^2}, \dots, \zeta_{1} =\zeta_{2}^p = \zeta^{p^{n-1}}$. This completes the proof.
\end{proof}
\begin{remark}
The expression of every primitive central idempotents of $F[C_{p^{n}}]$ has $p^n$ terms. Such an expression is a product of $n$ factors, and each factor containing $p$
terms. So each primitive central idempotent in $F[C_{p^{n}}]$ has an expression containing $pn$ terms.
\end{remark}
\section{Representations and Primitive Central idempotents of a Cyclic Group}\label{cypi}
In this section, we construct the irreducible representations of a cyclic group $G$ over a field $F$ of characteristic either $0$ or prime to $|G|$ 
and compute the primitive central idempotents of $F[G]$.
\subsection{Representations of Cyclic Groups}
\begin{thm}\label{theo}
Let $G = C_{n} = \langle{x | x^{n} = 1}\rangle$. Let $F$ be a field of characteristic $0$ or prime to $n$. Let $X^{n} - 1 = \displaystyle \prod^{k}_{i = 1}f_{i}(X)$ be the decomposition into irreducible polynomials over $F$. Then we have the following.
\begin{itemize} 
\item[1.] The set of all isomorphism types of irreducible $F$-representations of $G$ is in a bijective correspondence with the set of all monic irreducible factors of $X^{n} - 1$.
\item[2.]
Let $\rho_{1}, \rho_{2}, \dots , \rho_{k}$ be the $k$ irreducible $F$-representations of $G$. They are defined by $\rho_{i}(x) = C_{f_{i}(X)}$, where $C_{f_{i}(X)}$ denotes the companion matrix of $f_{i}(X)$ over $F$.
\end{itemize}
\end{thm}
\begin{proof} 
$(1)$ Consider the map $\phi: F[X] \longrightarrow F[G]$ given by: $f(X) \mapsto f(x)$. As $\phi$ is a ring epimorphism, $F[G] \simeq \frac{F[X]}{\langle{Ker(\phi)}\rangle}$. 
So, we get  
$$F[G] \simeq \frac{F[X]}{\langle X^{n} - 1 \rangle} = \displaystyle \frac{F[X]}{\langle \prod^{k}_{i = 1}f_{i}(X) \rangle}.$$
Since $F$ is a field of characteristic $0$ or prime to $n$, $X^{n} - 1$ is a separable polynomial with distinct roots.
Using Chinese Remainder Theorem, we can write
$$F[G] \simeq  \displaystyle \bigoplus^{k}_{i = 1}\frac{F[X]}{\langle f_{i}(X) \rangle}.$$
Under this isomorphism, the generator $x$ is mapped to the element 
$$(X+ \langle{f_{1}(X)}\rangle, X + \langle{f_{2}(X)}\rangle, \dots , X + \langle {f_{k}(X)}\rangle).$$ 
For each $i$, $1 \leq i \leq k$, $f_{i}(X)$ is an irreducible polynomial, so $\frac{F[X]}{\langle{f_{i}(X)}\rangle}$ is a field.
So, $G$ has $k$  irreducible representations.

\vskip2mm\noindent
$(2)$ Let $\rho_{i}$ be the irreducible representation corresponding to the simple component $\frac{F[X]}{\langle{f_{i}(X)}\rangle}$. Then $\rho_{i}(x)$: $\frac{F[X]}{\langle{f_{i}(X)}\rangle} \longrightarrow \frac{F[X]}{\langle{f_{i}(X)}\rangle}$ is defined by multiplication by $[X]$ $=$  $X + \langle{f_{i}(X)}\rangle$. Let $d$ be the degree of $f_{i}(X)$. Then $\{ [1], [X], \dots , [X^{d - 1}] \}$ is an ordered basis of $\frac{F[X]}{\langle{f_{i}(X)}\rangle}$, and w.r.t. this basis the matrix for $\rho_{i}(x)$ is $C_{f_{i}(X)}$, where $C_{f_{i}(X)}$ denotes the companion matrix of $f_{i}(X)$. 
This completes the proof.
\end{proof}
\begin{thm}\label{facy}
Let $G = C_{n} = \langle x \,|\, x^{n} = 1 \rangle$. Let $F$ be a field of characteristic either $0$ or prime to $n$. Let $\Phi_{n}(X)$ denote the $n$-th cyclotomic polynomial over $F$. Let $\Phi_{n}(X) = \displaystyle \prod^{k} _{i = 1}{f_{i}(X)}$ be the factorization into irreducible polynomials over $F$. Then we have the following.
\begin{itemize}
\item [1.] The set of all isomorphism types of faithful irreducible $F$-representations of $G$ is in a bijective correspondence with the set of all monic irreducible factors of $\Phi_{n}(X)$.
\item [2.] Let $\rho_{i}$ be the faithful irreducible $F$-representation corresponding to the irreducible factor $f_{i}(X)$. Then for each $i$, $\rho_{i}$ is defined by $x \mapsto C_{f_{i}(X)}$, where $C_{f_{i}(X)}$ denotes the companion matrix of $f_{i}(X)$, and the degree of $\rho_{i}$'s are same.
\end{itemize}
\end{thm}
\begin{proof}
$(1)$ Let $\overline{F}$ be the algebraic closure of $F$. Then 
	$$\overline{F}[G] \simeq \frac{\overline{F}[X]}{\langle { X^{n} - 1} \rangle} \simeq \displaystyle{\bigoplus_{i = 0}^{n-1} \frac{\overline{F}[X]}{\langle X - \zeta^{i}_{n} \rangle}},$$ where $\zeta_{n}$ is a primitive $n^{th}$ root of unity in $\overline{F}$. 
	Notice that the simple components corresponding to the faithful irreducible $\overline{F}$-representations of $G$ are  $\frac{\overline{F}[X]}{\langle X - \zeta^{i}_{n} \rangle}$, $(i, n) = 1$. 
	 Let $\eta_{i}$ be the representation corresponding to the component $\frac{\overline{F}[X]}{\langle X - \zeta^{i}_{n} \rangle}$. Let $^\sigma{\eta_{i}}(g) = \sigma({\eta_{i}}(g))$, for all $g \in G$ and ${\sigma \in  \mathrm{Gal}(F(\eta_{i})/F)}$.
	Then for each $i$, $\displaystyle  \oplus_{\sigma \in  \mathrm{Gal}(F(\eta_{i})/F)}{^\sigma{\eta_{i}}}$ is an irreducible $F$-representation. As for each $i$, $^\sigma{\eta_{i}}$ is a faithful irreducible $\overline{F}$-representation, $\displaystyle  \oplus_{\sigma \in  \mathrm{Gal}(F(\eta_{i})/F)}{^\sigma{\eta_{i}}}$ is a faithful irreducible $F$-representation. Also for each $i$, $f_{i}(X)$ is a separable polynomial, which implies $|\mathrm{Gal}(F(\eta_{i})/F)| = [F(\eta_{i}):F] = \mathrm{deg}{f_{i}(X)}$. So $G$ has at least $k$ faithful irreducible $F$-representations. Let $\rho_{1}, \rho_{2}, \dots , \rho_{k}$ be those faithful irreducible $F$-representations of $G$. 
	It is clear that for each $i$, $\rho_{i}$ is determined by the irreducible factor $f_{i}(X)$ of $\Phi_{n}(X)$. It remains to show that these are all the faithful $F$-irreducible representations of $G$. If not, let $\rho$ be a faithful irreducible $F$-representation, which is different from $\rho_{1}, \rho_{2}, \dots , \rho_{k}$. Then $\rho\otimes_{{F}} {\overline{F}}$ is also a faithful $\overline{F}$-representation of $G$, and is the direct sum of some faithful $\overline{F}$-irreducible representations. So, some of the $\eta_{i}$'s must occur in the decomposition, and which is a contradiction. Thus $\rho_{1}, \rho_{2}, \dots , \rho_{k}$ are the only faithful irreducible $F$-representations of $G$. This completes the proof.\\
	$(2)$ Follows from Theorem $\ref{theo}$.
\end{proof}
\subsection{Primitive Central idempotents of Cyclic Groups}
Let $G = \langle {x | x^{n} = 1}\rangle$ be a cyclic group of order $n$. Let $F$ be a field of characteristic $0$ or prime to $n$. Then $F[G]$ is isomorphic to $F[X]/\langle X^{n} - 1 \rangle$. Let $X^{n} - 1 = \prod_{d | n}{\Phi_{d}(X)}$ be the ``universal" decomposition into cyclotomic polynomials. For $\mathbb{Q}$, it is the factorization into monic irreducible polynomials. In general, $\Phi_{d}(X)$ will decompose further. Let $X^{n} - 1 = (X-1)f_{2}(X)\dots f_{k}(X)$, with $f_{1}(X) = X - 1$, be the factorization into monic irreducible polynomials over $F$. Note that each $f_{i}(X)$ is a monic irreducible factor of $\Phi_{d}(X)$ for some $d | n$. Then $F[G]$ is abstractly isomorphic to the direct sum of $F$-algebras $F[X]/\langle f_{i}(X) \rangle, i = 1, 2, \dots , k$. For a fixed $i$, let $\zeta_{1}, \zeta_{2}, \dots , \zeta_{d}$ be the roots of $f_{i}(X)$ in $\overline{F}$. Then the primitive central idempotent corresponding to $f_{i}(X)$ is $\displaystyle \sum^{k}_{j = 1}e_{\zeta^{-1}_{j}x}$, where $e_{X} = (1 + X + \dots + X^{n-1})/n$ and $X$ is an indeterminate. By Newton's formulae, power sums can be expressed in terms of elementary symmetric functions,  which implies the coefficients of $x, x^2, \dots , x^{n-1}$ can be expressed in terms of the coefficients of the $f_{i}(X)$. So, computation of the primitive central idempotents of $F[G]$ reduces to factorization of $X^{n} - 1$ into irreducible polynomials over $F$.
\section{Representations and Primitive Central Idempotents of an Abelian Group}
Let $A$ be a finite dimentional semisimple $F$-algebra. By Artin-Wedderburn structure theorem on semisimple algebras, we have $A$ is abstractly isomorphic to direct sum of matrix algebras over finite dimentional division algebras over $F$. We define the {\it abelian} part of $A$ as the inverse image of direct sum of commutative simple components. We now state the following two propositions without proof 
\begin{prop}\rm
Let $G$ be a finite group. Let $H \leq G$. Let $F$ be a field of characteristic does not divide $|G|$. Let $e_{H} = \frac{1}{|H|}\sum_{h \in H}{h}$. Then $e_{H}$ is an idempotent in $F[G]$. Moreover, if $H \vartriangleleft G$, then $e_{H}$ is a central idempotent in $F[G]$.  
\end{prop}
For the proof (see \cite{milies126}, Lemma $3.3.6$).
\begin{prop}\label{pro2}
Let $G$ be a finite group. Let $F$ be a field of characteristic $0$ or prime to $|G|$. Let $\Delta(G, G^{'}) = \sum_{ g \in G^{'}}{\alpha(g - 1)}$, where $\alpha \in F$. Then $F[G] = F[G]e_{G^{'}} \oplus \Delta (G, G^{'})$, where $F[G]e_G^{'} \cong F[G/G^{'}]$ is the sum of all commutative simple components of $F[G]$ and $\Delta (G, G^{'})$ sum of all the others.
\end{prop}
For the proof (see \cite{milies126}, Proposition $3.6.11$).
\subsection{Representations of abelian groups}
Let $G$ be a finite group. Let $F$ be a field of characteristic $0$ or prime to $|G|$. Let $\Omega_{G, F}$ be the set of all the inequivalent irreducible $F$-representations of $G$. Let $\Omega^{0}_{G} = \{\rho \in \Omega_{G, F} | $\textrm{ the irreducible $F$-representation $\rho$ of $G$ s.t. Im$\rho$ is cyclic }$\}$. We divide the set $\Omega_{G, F}$ into two parts:
\begin{enumerate}
\item The irreducible $F$-representations with image is abelian.
\item The irreducible $F$-representations with image is non-abelian.
\end{enumerate}
Let $\Omega^{ab}_{G} = \{\rho \in \Omega_{G, F} | $\textrm{ the irreducible $F$-representation $\rho$ of $G$ s.t. Im$\rho$ is abelian}$\}$. 
\begin{prop}\label{ab1}
$\Omega^{o}_{G} = \Omega^{ab}_{G}$.
\end{prop}
\begin{proof}
Let $\rho$ be an irreducible $F$-representation of $G$ with abelian image. Then $\rho$ factors through an ireeducible $F$-representation of $G/G^{'}$. By Proposition\ref{pro2}, Im$\rho$ is finite subgroup of a field, implies that Im$\rho$ is cyclic. 
So, $\Omega^{ab}_{G} \subseteq \Omega^{0}_{G}$. Clearly, if $\rho \in \Omega^{0}_{G}$, then Im$\rho$ is abelian, which implies that $\Omega^{0}_{G} \subseteq \Omega^{ab}_{G}$. This completes the proof.
\end{proof}
\begin{prop}\label{ab}
$\Omega^{0}_{G}$ parametrizes all the irreducible representation of $G/G^{'}$.
\end{prop}
\begin{proof}\rm
Let $\rho \in \Omega^{0}_{G}$. Then $\rho$ factors through an irreducible $F$-representation of $G/G^{'}$.
On the other hand, let $\rho$ be an irreducible $F$-representation of $G/G^{'}$, by Proposition\ref{ab1}, $\rho$ induces an irreducible $F$-representation of $G$ with cyclic image. Clearly, this is a bijective correspondence. This completes the proof.
\end{proof}
\begin{thm}
Let $G$ be a finite group of order $N$. 
For a divisor $d$ of $N$, let $H_{d} = \{ H \leq G | G/H \cong C_d \}$, and let $\Omega^{0}_{d} = \{ \textrm{faithful irreducible F-representations of $C_d$} \}$. Then
$\displaystyle{\cup_{d | N}(H_{d} \times \Omega^{0}_{d})}$ is in a bijective correspondence with $\Omega^{0}_{G}$.
\end{thm}
\begin{proof}
Let $\rho \in \Omega^{0}_{G}$. Then $\rho$ factors through a faithful irreducible $F$-representation of cyclic quotient. So, $\rho$ corresponds to an element in $\displaystyle{\cup_{d | N}({H_{d}} \times \Omega^{0}_{d})}$. On the other hand, take an element in $\displaystyle{\cup_{d | N}({H_{d}} \times \Omega^{0}_{d})}$, then corresponding to that there is an element in $\Omega^{0}_{G}$. Clearly, this correspondence is bijective. This completes the proof. 
\end{proof}
\begin{cor}\label{coro}
Let $G$ be a finite abelian group. 
Then $\Omega_{G, F} = \Omega^{0}_{G}$.
\end{cor}
\begin{proof}
By Proposition\ref{ab1}, $\Omega^{0}_{G} = \Omega^{ab}_{G}$. As $G$ is abelian, $\Omega_{G, F} = \Omega^{ab}_{G}$, which implies that $\Omega_{G, F} = \Omega^{0}_{G}$.
\end{proof}


\subsection{Primitive central idempotents of abelian group algebras}
\begin{definition}\rm
Let $N \vartriangleleft G$. Let $\phi : F[G] \longmapsto F[G/N]$ be the projection map. Let $\overline{e} \in F[G/N]$, then $e \in F[G]$ is called a {\it lift} of the element $\overline{e}$ if $\phi(e) = \overline{e}$, and $e_{N}e \in F[G]$ is called the {\it pull back} of $\overline{e}$.
\end{definition}

Let $G$ be an abelian group. By Corollary\ref{coro}, we get that every primitive central idempotent of $F[G]$ is pull back of the primitive central idempotent in $F[G/K]$ corresponding to a {\it faithful} irreducible representation of $G/K$, where $G/K$ is isomorphic to a cyclic group. Let $G/K = \langle{ \overline{x} | {\overline{x}}^{n} = 1 }\rangle \simeq C_{n}$, where $\overline{x} = xK$. Let $e = \sum_{i = 0}^{n-1}{\alpha_{i}{\overline{x}}^{i}}$ be a primitive central idempotent in $F[G/K]$ corresponding to a {\it faithful} irreducible representation of $G/K$. Then the pull back of $e = \sum_{i = 0}^{n-1}{\alpha_{i}{\overline{x}}^{i}}$, that is, $e_{K}\sum_{i = 0}^{n-1}{\alpha_{i}{{x}}^{i}}$, where $e_{K} = \frac{1}{|K|}\sum_{k \in K}k$, is a primitive central idempotent of $F[G]$. Later using long generators, we shall give further factorization of $e_{K}\sum_{i = 0}^{n-1}{\alpha_{i}{{x}}^{i}}$. Thus the problem reduces to computing the primitive central idempotents of a cyclic group as we have described in Section\ref{cypi}.
\section{Representations of a finite solvable group over $\mathbb{C}$}
The following theorem is known, we call this theorem index-$p$ theorem. Using this theorem and a long presentation, one can iductively construct the irreducible representations of a finite solvable group over $\mathbb{C}$.
\begin{thm}[\cite{alg111}, Theorem $13.52$] (Index-$p$ theorem)\label{index}
Let $G$ be a group, $H$ a normal subgroup of index $p$, 
$p$ a prime. 
Let $\eta$ be an irreducible representation of $H$ over $\mathbb{C}$.
\begin{enumerate}
\item If the $G/H$-orbit of $\eta$ is a singleton, then $\eta$ extends
to $p$ mutually inequivalent representations 
$\rho_1, \rho_2, \ldots, \rho_p$  of $G$.
\item If the $G/H$-orbit of $\eta$ consists of $p$ points 
$\eta = \eta_1, \eta_2,\ldots, \eta_p$
then the induced representations $\eta_1\uparrow_H^G$, 
$\eta_2\uparrow_H^G,\ldots, \eta_p\uparrow_H^G$, are equivalent,
say $\rho$ and $\rho$ is irreducible.
\end{enumerate}
\end{thm}
\begin{cor}
Let $G$ be a solvable group. 
Then $G$ has a maximal 
	subnormal series such that the successive quotients are 
	isomorphic to cyclic groups of prime order. The last but one term in 
	the maximal subnormal series is a cyclic group of prime order. 
	So all its irreducible 
	$F$-representations are of degree one. So starting with degree one 
	representations of subgroups, we can build all the irreducible representations of $G$ by the processes of extension and induction.
\end{cor}
\section{Primitive central idempotents of complex group algebra of a finite solvable group}
The following theorem is due to Berman (see \cite{berman120}). Here, we give an elegant proof of the theorem. For a finite solvable group $G$, using the following theorem and a long system of  generators, one can iductively construct the primitive central idempotents of $\mathbb{C}[G]$. 
 \begin{thm}\label{berman}
Let $G$ be a finite group and $H$ be a normal subgroup of index $p$, a prime. Let $G/H = \langle {\overline{x}} \rangle$, for some $x$ in $G$. Let $\overline{C}(x)$ be the conjugacy class sum of $x$ in $\mathbb{C}[G]$. Let $(\eta, W)$ be an irreducible representation of $H$ over $\mathbb{C}$ and $e_{\eta}$ be its corresponding primitive central idempotent in $\mathbb{C}[H]$. 
We distinguish two cases:
\begin{itemize}
\item[(1)] If $e_{\eta}$ is a central idempotent in $\mathbb{C}[G]$, 
then $\eta$ extends to $p$ distinct irreducible representations
$\rho_{0}, \rho_{1}, \dots , \rho_{p-1}$ (say) of $G$. 
Moreover, $x$ can be chosen s.t. $(\overline{C}(x))^{p}{e_{\eta}} = {\lambda}{e_{\eta}}$, where $\lambda \neq 0$.
For each $i$, 
$0 \leq i \leq p-1$, let $e_{\rho_{i}}$ be the primitive central idempotent corresponding to the representation $\rho_{i}$. Then
$${e_{\rho_{i}} = \frac{1}{p} 
\Big{(} 1 + \zeta^ic + \zeta^{2i} c^2 +\cdots +\zeta^{i(p-1)}c^{p-1}\Big{)} 
e_{\eta},}$$
where $c = \frac{\overline{C}{(x)}{e_{\eta}} }{\sqrt[p]{\lambda}}$ and $\zeta$ 
is a primitive $p^{th}$ root of unity in $F$.  
Moreover, 
$$e_{\eta} = e_{\rho_0} + e_{\rho_1} + \cdots + e_{\rho_{p-1}}.$$

\item[(2)] If $e_{\eta}$ is not a central idempotent in $\mathbb{C}[G]$,
then 
$\eta\uparrow_H^G, \eta^{x}\uparrow_H^G, \ldots ,\eta^{x^{p-1}}\uparrow_H^G$ 
are all equivalent to an irreducible representation $\rho$ (say) of $G$ over $\mathbb{C}$ and in this case,
$$e_{\rho} = e_{\eta} + e_{\eta^{x}} + \cdots + e_{\eta^{x^{p-1}}}.$$
\end{itemize}
\end{thm}
\begin{proof}
 $(1)$ Since $e_{\eta}$ is a central 
idempotent in $\mathbb{C}[G]$, which implies that $\eta \cong \eta^{x}$. By Theorem $\ref{index}$, $\eta$ extends to $p$ distinct irreducible representations $\rho_{0}, \rho_{1}, \dots ,$ $\rho_{p-1}$ (say) of $G$. By Frobenius reciprocity theorem, the induced representation 
${\eta}\uparrow^{G}_{H} \cong \rho_{0} \oplus \rho_{1} \oplus \dots \oplus \rho_{p-1}$, which implies that 
$\mathbb{C}[G]e_{\eta} = \bigoplus^{p-1}_{i = 0} \mathbb{C}[G]e_{\rho_{i}}$ and $e_{\eta} = \sum^{p-1}_{i = 0}e_{\rho_{i}}$. Note that $\mathbb{C}[G]e_{\eta}$ is a ring with identity $e_{\eta}$. Notice that $\mathbb{C}[G]e_{\eta}$ is the direct sum of $p$ minimal two-sided ideals $\mathbb{C}[G]e_{\rho_{i}}, i = 0,1, \dots , p-1$, and therefore $\mathbb{C}[G]e_{\eta}$ contains $p$ primitive central idempotents. 

Since $G/H=\langle xH\rangle$, for some $x\in G$, this implies that $x^p\in H$.
Then one can show that $(\overline{C}_G(x))^p=x^pa$ for some 
$a\in \mathbb{C}[H]$. It follows that $(\overline{C}_G(x))^p$ is central element in
$\mathbb{C}[H]$. Hence, by Schur's Lemma,
$$(\overline{C}_G(x))^pe_{\eta}=\lambda e_{\eta}, \hskip5mm \textrm{for some}\, \lambda \in
\mathbb{C}.$$
Note that $\lambda$ depends on chosen $x$.
\vskip2mm\noindent
{\it Claim:} The element $x\in G-H$ can be chosen in such a
way that $\lambda \neq 0$.
\vskip2mm\noindent
To prove the claim, let $\rho$ be an extension of $\eta$, with corresponding
character $\chi_{\rho}$ and the primitive central idempotent $e_{\rho}$.
If $\chi_{\rho}$ vanishes on $G-H$, then
$$e_{\rho} = \frac{1}{|G|}\sum_{g\in G}
\chi_{\rho}(g^{-1})g = \frac{1}{p}\frac{1}{|H|}\sum_{h\in H}
\chi_{\eta}(g^{-1})g = \frac{1}{p}e_{\eta},$$
and this implies that $e_{\rho}$ is {\it not} an idempotent, a contradiction. Thus there exists
$x\in G-H$ such that $\chi_{\rho}(x)\neq 0$.

If $V$ is the representation
space corresponding to $\eta$, then it is also a representation space for the
extension $\rho$. Extend $\rho:G\rightarrow {\rm GL}(V)$ linearly  to algebra
homomorphism  $\mathbb{C}[G]\rightarrow {\rm End}(V)$, which we still denote by
$\rho$. Similarly, extend $\eta$ to the algebra homomorphism
$\eta:\mathbb{C}[H]\rightarrow {\rm End}(V)$.
Now by Schur's Lemma, $\rho(\overline{C}_G(x))=\mu I$ for
some $\mu \in \mathbb{C}$. Taking traces of both sides, we get
$$|C_G(x)|\chi_{\rho}(x)=\mu \deg\rho.$$
Since $\chi_{\rho}(x)\neq 0$, we get $\mu \neq 0$. Then
$$\lambda I = \rho(\lambda e_{\eta})=\rho((\overline{C}_G(x))^pe_{\eta})=
\rho( (\overline{C}_G(x) e_{\eta})^p)=\mu^pI\neq 0,$$
hence $\lambda \neq 0$. This proves the Claim.

For $i=0,1, \dots, p-1$, let 
$$f_{i} = \frac{1}{p} 
\Big{(} 1 + \zeta^ic + \zeta^{2i} c^2 +\cdots +\zeta^{i(p-1)}c^{p-1}\Big{)}e_{\eta},$$
where $c = \frac{\overline{C}_{G}(x){e_{\eta}} }{\sqrt[p]{\lambda}}$ and $\zeta$ 
is a primitive $p^{th}$ root of unity in $\mathbb{C}$. It is clear that all $f_{i}$'s are non zero. Since $e_{\eta}$ and $\overline{C}_{G}(x)$ are central in $\mathbb{C}[G]$, then all $f_{i}$'s are central in $\mathbb{C}[G]$.
It is clear that
\begin{equation*}
c^pe_{\eta}=e_{\eta}.\tag{*}
\end{equation*}
Suppressing the index $i$, let us write
$$f=\frac{1}{p}(1+\zeta c + \cdots + \zeta^{p-1}c^{p-1})e_{\eta}.$$
Then
\begin{align*}
cf 
&=\frac{1}{p} (ce_{\eta} +\zeta c^2e_{\eta} + \cdots +
\zeta^{p-2}c^{p-1}e_{\eta} + \zeta^{p-1}c^pe_{\eta})\\
&=\frac{1}{p} (ce_{\eta} +\zeta c^2e_{\eta} + \cdots +
\zeta^{p-2}c^{p-1}e_{\eta} + \zeta^{-1}e_{\eta}) \mbox{ (by (*))}\\
&= \frac{\zeta^{-1}}{p}(1+\zeta c + \cdots + \zeta^{p-1}c^{p-1})e_{\eta}\\
&=\zeta^{-1}f.
\end{align*}
Thus for $0\leq i,k<p$, we can deduce from above calculation that
\begin{equation*}
c^kf_i=\zeta^{-ik}f_i.\tag{**}
\end{equation*}
Now we show that $f_if_j=\delta_{i,j}f_j$.
\begin{align*}
f_if_j &=\frac{1}{p}\Big{(} \sum_{k=0}^{p-1} (\zeta^ic)^k \Big{)}e_{\eta}
f_j\\
&=\frac{1}{p} \Big{(}\sum_{k=0}^{p-1} \zeta^{ik} (c^kf_j)\Big{)}e_{\eta}\\
&=\frac{1}{p} \Big{(}\sum_{k=0}^{p-1}\zeta^{ik}\zeta^{-jk} f_j
\Big{)}e_{\eta}\hskip5mm (\mbox{by (**)})\\
&=\frac{1}{p} \Big{(}\sum_{k=0}^{p-1}\zeta^{(i-j)k}
\Big{)}f_je_{\eta}\\
&=\delta_{i,j}(f_je_{\eta})=\delta_{i,j}f_j.
\end{align*}
We now show that $\sum^{p-1}_{i = 0} {f_{i}} = e_{\eta}$. 
\begin{align*}
\sum^{p-1}_{i = 0} {f_{i}} & = \sum_{i=0}^{p-1}\Big{(} \frac{1}{p}  \sum_{k=0}^{p-1} (\zeta^ic)^k \Big{)}e_{\eta}\\
& = \frac{1}{p} 
\Big{\{} p + \Big{(}\sum^{p-1}_{i = 0}\zeta^{i}\Big{)}c + 
\Big{(}\sum^{p-1}_{i = 0}\zeta^{2i}\Big{)}c^{2} +  \cdots + \Big{(}\sum^{p-1}_{i = 0}\zeta^{(p-1)i}\Big{)}c^{p-1}\Big{\}}e_{\eta}\\
& = e_{\eta}
\end{align*}
Notice that each $f_{{i}}$ belongs to $\mathbb{C}[G]e_{\eta}$. 
Therefore, $\{f_{i}\,|\, i = 0,1, \dots , p-1\}$ are mutually pairwise orthogonal central idempotents of the ring $\mathbb{C}[G]e_{\eta}$, and whose sum is $e_{\eta}$. Since $\mathbb{C}[G]e_{\eta}$ contains $p$ primitive central idempotents, $\{f_{i} \,|\, i = 0,1, \dots , p-1\}$ are all the primitive central idempotents of $\mathbb{C}[G]e_{\eta}$. Hence, for each $i = 0,1, \dots , p-1$,
$${e_{\rho_{i}} = \frac{1}{p} 
	\Big{(} 1 + \zeta^ic + \zeta^{2i} c^2 +\cdots +\zeta^{i(p-1)}c^{p-1}\Big{)} 
	e_{\eta},}$$
where $c = \frac{\overline{C}_{G}(x){e_{\eta}} }{\sqrt[p]{\lambda}}$ and $\zeta$ 
is a primitive $p^{th}$ root of unity in $\mathbb{C}$.  
Moreover, 
$$e_{\eta} = e_{\rho_0} + e_{\rho_1} + \cdots + e_{\rho_{p-1}}.$$ This completes the proof of $(1)$. 

$(2)$ If $e_{\eta}$ is not a central idempotent in $\mathbb{C}[H]$, then $\eta$ is not
equivalent to $\eta^{x}$. By Theorem \ref{index} 
induced representations $\eta\uparrow_H^G, \eta^{x}\uparrow_H^G,
\ldots ,\eta^{x^{p-1}}\uparrow_H^G$ 
are all equivalent to $\rho$ (say) and $\rho$ is irreducible.
Since the character $\chi_{\rho}$ of $\rho$ vanishes outside
normal subgroup $H$ and $\chi_{\rho}\downarrow^{G}_{H} = \chi_{\eta} + \chi_{\eta^{x}} + \ldots + \chi_{\eta^{x^{p-1}}}$,
we have $e_{\rho} = e_{\eta} + e_{\eta^{x}} + \cdots + e_{\eta^{x^{p-1}}}.$
This completes the proof of $(2)$.
\end{proof}
\begin{Example}(Special linear group $\mathrm{SL}_2(3)$)
Using a long presentation of $\mathrm{SL}_2(3)$, we construct the irreducible representations of $\mathrm{SL}_2(3)$ over $\mathbb{C}$ and the primitive central idempotents of $\mathbb{C}[\mathrm{SL}_2(3)]$.
Let us consider the maximal subnormal series
\begin{equation}\label{long}
\langle e \rangle = G_{o} < C_2 = G_1 < C_4 = G_2 < Q_{8} = G_3 < \mathrm{SL}_2(3) = G_4.
\end{equation}
A long presentation of ${\rm SL}_2(3)$ associated with the series \ref{long} is
\begin{align*}
\langle x,y,z,t \,|\, x^{2} = 1, y^{2} = x, y^{2} = z^{2}, z^{-1}yz = xy, t^{3} = 1, t^{-1}yt = z, t^{-1}zt = yz \rangle.
\end{align*}
The group $Q_8 = G_{3}$ is a normal subgroup of index $3$ in $\mathrm{SL}_2(3)$. Now we construct the complex irreducible representations of $\mathrm{SL}_2(3)$ starting with the complex irreducible representations of $Q_{8}$. 

Let $\rho_1$, $\rho_2$, $\rho_3$, $\rho_4$, $\rho_5$ be the five inequivalent complex irreducible representations
 $Q_{8}$, and are defined by
 \begin{center}
 \begin{tabular}{lcl}
  $\rho_1(y) = 1$, & & $\rho_1(z) = 1$,\\
  $\rho_2(y) = 1$, & & $\rho_1(z) = -1$,\\
  $\rho_3(y) = -1$, & & $\rho_1(z) = 1$,\\
  $\rho_4(y) = -1$, & & $\rho_1(z) = -1$,
 \end{tabular}
\end{center}
and
\begin{equation*}
[\rho_{5}(y)] = 
\begin{bmatrix}
i & 0 \\
0 & -i \end{bmatrix},
\hskip3mm 
[\rho_{5}(z)] = 
\begin{bmatrix}
0 & -1 \\
1 & 0 
\end{bmatrix}.
\end{equation*}
 The trivial representation $\rho_1$ extends to three irreducible representations of $\mathrm{SL}_2(3)$, and are given by 
$$\theta_1(t) = 1, \hskip3mm \theta_2(t) = \omega, \hskip3mm \theta_{3}(t) = \omega^{2}, \hskip5mm \mathrm{where} \, \omega = e^{2\pi i/3}.$$ 

Notice that $\rho_2,\rho_3, \rho_4$ are conjugate to each other, so they induce the same irreducible representation $\theta_4$ of degree 3, and is given by
\begin{center}
	$[\theta_4(x)] = I_{3},  \hskip3mm
	[\theta_4(y)] =
	\begin{bmatrix}
	1 & 0 & 0\\
	0  & -1  & 0\\
	0 & 0 & -1
	\end{bmatrix}, \hskip3mm
	[\theta_4(z)] = 
	\begin{bmatrix}
	-1 & 0 & 1\\
	0  & 1  & 0\\
	0 & 0 & -1 
	\end{bmatrix},$
\end{center}
\begin{center}
	$[\theta_4(t)] = 
	\begin{bmatrix}
	0& 0 & 1\\
	1  & 0  & 0\\
	0 & 1 & 0
	\end{bmatrix}$.
\end{center}
$\rho_{5}$ extends to three irreducible representations of $\mathrm{SL}_2(3)$, let these extensions be $\theta_{5}, \theta_{6}, \theta_{7}$. 
The matrix representations of $\theta_{5}, \theta_{6}, \theta_{7}$ are given by
\begin{equation*}
\theta_k(x)= 
\begin{bmatrix}
-1 & 0 \\
0 & -1
\end{bmatrix}, \hskip1mm 
\theta_k(y)= 
\begin{bmatrix}
i & 0 \\
0 & -i 
\end{bmatrix},  \hskip1mm
\theta_k(z)= 
\begin{bmatrix}
0 & -1 \\
1 & 0 
\end{bmatrix} \hskip3mm (k = 5,6,7),
\end{equation*}
and 
\small
\begin{equation*}
\theta_{5}(t) =\begin{bmatrix}
\frac{-1+i}{2} & \frac{-1 - i}{2}\\
\frac{1-i}{2} & \frac{-1-i}{2}\end{bmatrix}, 
\theta_{6}(t) = \omega\begin{bmatrix}
\frac{-1+i}{2} & \frac{-1 - i}{2}\\
\frac{1-i}{2} & \frac{-1-i}{2}\end{bmatrix}, 
\theta_{7}(t) = \omega^{2}\begin{bmatrix}
\frac{-1+i}{2} & \frac{-1 - i}{2}\\
\frac{1-i}{2} & \frac{-1-i}{2}\end{bmatrix}.
\end{equation*}
The primitive central idempotents of $\mathbb{C}[Q_{8}]$ are 

\begin{enumerate}
\item $e_xe_ye_z,$ 
\item $e_xe_ye_{-z},$ 
\item $e_xe_{-y}e_z,$
\item $e_xe_{-y}e_{-z},$
\item $e_{-x}e_{iy} + e_{-x}e_{-iy} = 1 - e_x,$
\end{enumerate}
where $e_X = \frac{1+X}{2}$; $X\in \{\pm\, x, \pm\, y,  \pm\,  z\}$.
 
Let $\overline{C}(t)$ denotes the conjugacy class sum of $t$ as an element of $\mathbb{C}[\mathrm{SL}_2(3)]$. Then

\begin{enumerate}
\item $u_1 =  \frac{1}{3}\begin{Bmatrix} 1 + \begin{pmatrix} \frac{\overline{C}(t)}{-2} \end{pmatrix} + \begin{pmatrix}\frac{\overline{C}(t)}{-2}\end{pmatrix}^{2}\end{Bmatrix}e_{-x},$
\item $u_2 = \frac{1}{3}\begin{Bmatrix} 1 + \omega^2 \begin{pmatrix} \frac{\overline{C}(t)}{-2} \end{pmatrix} + 
\omega\begin{pmatrix}\frac{\overline{C}(t)}{-2}\end{pmatrix}^{2}\end{Bmatrix}e_{-x}, $
\item $u_3 =  \frac{1}{3}\begin{Bmatrix} 1 + \omega \begin{pmatrix} \frac{\overline{C}(t)}{-2} \end{pmatrix} + \omega^2\begin{pmatrix}\frac{\overline{C}(t)}{-2}\end{pmatrix}^{2}\end{Bmatrix}e_{-x},$
\item $u_4 = e_xe_ye_{-z}+ e_xe_{-y}e_z + e_xe_{-y}e_{-z}, $
\item $u_{5} = e_xe_ye_ze_{wt},$
\item $u_{6} = e_xe_ye_ze_{w^{2}t},$
\item $u_{7} = e_xe_ye_ze_t,$
\end{enumerate}
where $\omega = e^{2\pi i/3}$ and $e_X = \frac{1+X}{2}$ for $X\in \{\pm\, x, \pm\, y,  \pm\,  z\}$; $e_Y=\frac{1+Y+Y^2}{3}$ for $Y \in \{t,\, \omega t,\, \omega^2 t\}$, are the primitive central idempotents of $\mathbb{C}[\mathrm{SL}_2(3)]$. 	
\end{Example}

\end{document}